\documentclass[11pt]{article}
\usepackage[a4paper,margin=1in]{geometry}
\usepackage[T1]{fontenc}
\usepackage[utf8]{inputenc}
\usepackage{lmodern}
\usepackage{amsmath,amssymb,amsthm,mathtools,bm}
\usepackage{microtype}
\usepackage{hyperref}
\usepackage[nameinlink,capitalize]{cleveref}
\usepackage{xcolor}
\usepackage{booktabs}
\usepackage{enumitem}
\hypersetup{colorlinks=true,linkcolor=blue!45!black,citecolor=blue!45!black,urlcolor=blue!45!black}
\numberwithin{equation}{section}

\newtheorem{theorem}{Theorem}[section]

\newtheorem{lemma}[theorem]{Lemma}
\newtheorem{corollary}[theorem]{Corollary}
\theoremstyle{remark}
\newtheorem{remark}[theorem]{Remark}
\crefname{theorem}{theorem}{theorems}
\crefname{lemma}{lemma}{lemmas}
\crefname{proposition}{proposition}{propositions}
\crefname{corollary}{corollary}{corollaries}

\newcommand{\R}{\mathbb R}

\newcommand{\Da}{\Delta_5}

\title{\textbf{Hardy Criticality in Axisymmetric Swirl}}
\author{Rishad Shahmurov}
\date{}

\begin{document}
\maketitle

\begin{abstract}
For smooth axisymmetric Navier--Stokes flow with swirl, set
\(F=u^\theta/r\), \(\Gamma=ru^\theta\), \(G=\omega^\theta/r\), and \(U=u^r/r\).
Weighted swirl energies are classical; we determine the sharp two-parameter coercivity spectrum of the exact power-weight family
\(\int |F|^p r^\alpha\,dr\,dz\), \(p>1\), \(\alpha>1\).
The full diffusion form has a positive gradient gap for \(\alpha<2p+1\), loses every uniform gap at \(\alpha=2p+1\), and is indefinite above it.  Exactly at the same threshold the radial-strain term cancels, and the identity becomes the classical \(L^p\) circulation energy for \(\Gamma\).  Thus stretching cancellation and Hardy criticality coincide within the power family.  We also characterize \(r^3drdz\) by conservative \(F\)-transport, self-adjoint five-dimensional diffusion, and a flat \(G\)-source pairing, and obtain the optimal strain form-bound threshold within the Hardy-coercive powers.
\end{abstract}

\noindent\textbf{Keywords.} axisymmetric Navier--Stokes; swirl; Hardy inequality; weighted energy; circulation; criticality.\\
\textbf{MSC 2020.} 35Q30, 76D05, 35B35, 35B65.

\section{Introduction}

Let
\[
 u=u^r(r,z,t)e_r+u^\theta(r,z,t)e_\theta+u^z(r,z,t)e_z
\]
be a smooth axisymmetric solution of the three-dimensional incompressible Navier--Stokes equations on \(\R^3\).  We suppress the harmless angular factor \(2\pi\) in all meridional integrals.  The variables
\[
 F=\frac{u^\theta}{r},\qquad \Gamma=ru^\theta=r^2F,
 \qquad G=\frac{\omega^\theta}{r},\qquad U=\frac{u^r}{r}
\]
separate three familiar features of the problem: the circulation \(\Gamma\) obeys a maximum principle, the swirl rate \(F\) is diffused by a five-dimensional Bessel operator but is stretched by \(U\), and the vorticity ratio \(G\) is forced by \(\partial_z(F^2)\).  These variables and their pole compatibility are standard in the axisymmetric theory; see, for example, \cite{ChaeLee2002,LiuWang2009,LeiZhang2017,ChenFangZhang2017,RencZaj2019,WangHuangWeiYu2024}.

Weighted swirl and circulation energies are also classical.  Chae and Lee \cite{ChaeLee2002} obtained, among other a priori estimates, the familiar \(L^p\) control of the circulation \(ru^\theta\) in the standard range \(p\ge2\).  Chen, Fang and Zhang \cite{ChenFangZhang2017} proved weighted angular-velocity criteria involving \(r^d u^\theta\) for \(0\le d<1\) and explicitly observed that their Hardy-based argument fails at \(d=1\).  Renc\l awowicz and Zaj\k aczkowski \cite{RencZaj2019} introduced the weighted circulation variable \(u_a=\Gamma/r^a\), derived its \(L^s\)-energy with the positive potential \(a(2-a)/r^2\), and used the scale-balanced choice \(as=3\).  Thus neither weighted swirl energies, the circulation endpoint, nor the appearance of a Hardy obstruction at the endpoint is new here.

The purpose of this paper is instead to close the algebraic and coercive classification of the entire two-parameter power family.  We ask: for fixed \(p>1\), what can the weight \(r^\alpha\) actually accomplish in the exact \(F\)- and \(G\)-energies, what is the best full-gradient coercive constant, and where does that constant fail?  The answer is rigid.  Write
\[
 d\mu_\alpha=r^\alpha\,dr\,dz,\qquad \alpha>1.
\]
For the \(F^p\)-energy, the coefficient of radial compression vanishes at one and only one weight,
\[
 \alpha=2p+1.
\]
At precisely the same weight the radial diffusion form reaches the sharp Hardy boundary.  Moreover, the resulting energy is exactly the ordinary three-dimensional \(L^p\)-norm of the circulation \(\Gamma\).  Thus the apparent freedom to tune a radial power weight is illusory: cancellation of stretching simply returns the problem to the circulation variable, and no positive \(F\)-gradient gap survives at that endpoint.

The second conclusion concerns the coupled \((F,G)\) system.  The weight \(r^3drdz\), corresponding to the usual five-dimensional lift, is selected independently by three exact properties: it is the unique power weight for which the \(F\)-equation is conservative, the Bessel operator is formally self-adjoint, and the source pairing in the \(G\)-energy becomes the flat nonsingular product \(2\int \Gamma W G\,drdz\), with \(W=\Gamma_z/r\).  Among the Hardy-coercive powers \(3\le\alpha<2p+1\) for the \(F^p\) family, and \(3\le\alpha<5\) for the \(G\)-family, the same weight also maximizes the scale-invariant strain form bound that can be absorbed by the Hardy gap.

The sharp constants follow from the one-dimensional weighted Hardy inequality, but their coincidence with the exact Navier--Stokes stretching coefficients yields a useful no-go principle and an optimal power-weight form-bound threshold.  This is naturally comparable with the scale-invariant form-boundedness viewpoint of Lei--Zhang \cite{LeiZhang2017}, although the present result concerns the exact multiplier spectrum rather than a global regularity theorem.  Weighted Hardy, weighted-swirl, and Hardy--Sobolev estimates have played an important role in axisymmetric regularity theory \cite{ChenFangZhang2017,RencZaj2019,Zhang2017,WangHuangWeiYu2024}.  The contribution here is the sharp closed phase diagram for the exact \((p,\alpha)\) power-multiplier family: the best full-gradient coercive constant, the equivalence of stretching cancellation with Hardy criticality, full-form indefiniteness beyond the critical line, and the exact conjugation of that line to the physical circulation energy.

\section{Axisymmetric equations and weighted calculus}

The physical meridional incompressibility relation is
\begin{equation}\label{eq:div3}
 \partial_r u^r+\frac{u^r}{r}+\partial_z u^z=0.
\end{equation}
The variables \(F,G\) satisfy
\begin{align}
 F_t+u^rF_r+u^zF_z
 &=\nu\Da F-2UF,\label{eq:F}\\
 G_t+u^rG_r+u^zG_z
 &=\nu\Da G+\partial_z(F^2),\label{eq:G}
\end{align}
where
\begin{equation}\label{eq:Delta5}
 \Da=\partial_{rr}+\frac3r\partial_r+\partial_{zz}.
\end{equation}
The circulation solves
\begin{equation}\label{eq:Gamma}
 \Gamma_t+u^r\Gamma_r+u^z\Gamma_z
 =\nu\left(\Gamma_{rr}-\frac1r\Gamma_r+\Gamma_{zz}\right).
\end{equation}
For smooth Cartesian axisymmetric fields one has \(\Gamma=O(r^2)\) and \(F,G=O(1)\) at the axis; the pole conditions needed to pass between cylindrical and Cartesian regularity are described in \cite{LiuWang2009}.
All weighted identities below are first justified for smooth rapidly decaying solutions.  More generally, they hold on any smooth time interval on which the displayed weighted quantities are finite, by inserting radial and axial cutoffs and passing to the limit.  The axis boundary terms vanish because of the Cartesian pole conditions.

\begin{lemma}\label{lem:weighteddiv}
For every \(\alpha\in\R\),
\[
 r^{-\alpha}\left[\partial_r(r^\alpha u^r)+\partial_z(r^\alpha u^z)\right]
 = (\alpha-1)U.
\]
Consequently,
\[
 \int (u^r h_r+u^z h_z)\,h\,d\mu_\alpha
 =-\frac{\alpha-1}{2}\int U h^2\,d\mu_\alpha
\]
whenever the boundary terms vanish.
\end{lemma}

\begin{proof}
The first identity follows immediately from \eqref{eq:div3}.  The second is integration by parts applied to \(\frac12 b\cdot\nabla(h^2)\), where \(b=(u^r,u^z)\).
\end{proof}

\begin{lemma}\label{lem:diffusion}
Let \(\alpha>1\), \(p>1\), and let \(f\) be smooth, axis-compatible, and sufficiently decaying. Put \(H=|f|^{p/2}\). Then
\begin{align}
 \int |f|^{p-2}f\,\Da f\,d\mu_\alpha
 ={}&-\frac{4(p-1)}{p^2}\int |\nabla H|^2\,d\mu_\alpha\nonumber\\
 &-\frac{(3-\alpha)(\alpha-1)}p
 \int H^2 r^{\alpha-2}\,dr\,dz.\label{eq:diffusion-general}
\end{align}
For \(1<p<2\), the left-hand side is understood through the standard convex renormalization described in the proof.
\end{lemma}

\begin{proof}
For \(p\ge2\), direct integration by parts gives
\[
 -(p-1)\int |f|^{p-2}|\nabla f|^2r^\alpha
 =-\frac{4(p-1)}{p^2}\int|\nabla H|^2r^\alpha.
\]
The mismatch between the coefficient \(3/r\) in \(\Da\) and the measure \(r^\alpha drdz\) is
\[
 (3-\alpha)\int |f|^{p-2}ff_r r^{\alpha-1}
 =-\frac{(3-\alpha)(\alpha-1)}p\int |f|^p r^{\alpha-2},
\]
and the axis term vanishes because \(\alpha>1\) and \(f\) is axis-compatible.

For \(1<p<2\), it is useful to avoid differentiating \(|f|^{p-2}f\) at the zero set.  Define an even convex function \(\Phi_\varepsilon\) by
\[
 \Phi_\varepsilon(0)=\Phi_\varepsilon'(0)=0,
 \qquad
 \Phi_\varepsilon''(s)=p(p-1)(s^2+\varepsilon^2)^{(p-2)/2}.
\]
Then \(0\le \Phi_\varepsilon(s)\le |s|^p\),
\(s\Phi_\varepsilon'(s)/p\le |s|^p\), and, as \(\varepsilon\downarrow0\),
\[
 \Phi_\varepsilon(s)\uparrow |s|^p,
 \qquad
 \frac1p\Phi_\varepsilon''(s)\uparrow (p-1)|s|^{p-2}
 \quad (s\ne0).
\]
Test the diffusion term with \(\Phi_\varepsilon'(f)/p\). Integration by parts gives exactly
\[
 -\frac1p\int \Phi_\varepsilon''(f)|\nabla f|^2r^\alpha
 -\frac{(3-\alpha)(\alpha-1)}p
   \int \Phi_\varepsilon(f)r^{\alpha-2}.
\]
Monotone convergence applies to the nonnegative gradient density and to \(\Phi_\varepsilon(f)\); the standard Sobolev chain rule identifies
\[
 (p-1)|f|^{p-2}|\nabla f|^2
 =\frac{4(p-1)}{p^2}|\nabla |f|^{p/2}|^2
\]
almost everywhere. This yields \eqref{eq:diffusion-general}.  The same renormalization, with spatial and temporal cutoffs if needed, justifies the \(F^p\)-testing below for the whole range \(p>1\).
\end{proof}

\begin{lemma}\label{lem:hardy}
For \(\alpha>1\) and \(h\in C_c^\infty((0,\infty)\times\R)\),
\begin{equation}\label{eq:hardy}
 \int h^2r^{\alpha-2}\,dr\,dz
 \le \frac4{(\alpha-1)^2}
 \int |h_r|^2r^\alpha\,dr\,dz.
\end{equation}
The constant is sharp.
\end{lemma}

\begin{proof}
For each fixed \(z\), integration by parts gives
\[
 (\alpha-1)\int_0^\infty h^2r^{\alpha-2}dr
 =-2\int_0^\infty hh_rr^{\alpha-1}dr.
\]
Cauchy--Schwarz gives
\[
 (\alpha-1)M^{1/2}\le 2R^{1/2},
 \qquad
 M=\int_0^\infty h^2r^{\alpha-2}dr,\quad
 R=\int_0^\infty |h_r|^2r^\alpha dr,
\]
which is \eqref{eq:hardy} after integration in \(z\).

We record an explicit sharpness sequence because it will also be used below.  Put
\(\beta=(\alpha-1)/2\).  Choose smooth cutoffs \(\chi,\psi\) with
\(\chi=0\) on \(0,1/2]\), \(\chi=1\) on \([1,\infty)\),
\(\psi=1\) on \([0,1]\), and \(\psi=0\) on \([2,\infty)\).  For \(R\ge4\), set
\[
 h_R(r)=r^{-\beta}\chi(r)\psi(r/R).
\]
On \([1,R]\) the function is exactly \(r^{-\beta}\); the two cutoff regions contribute only \(O(1)\).  Hence
\begin{align*}
 \int_0^\infty h_R^2r^{\alpha-2}dr&=\log R+O(1),\\
 \int_0^\infty |h_R'|^2r^\alpha dr&=\beta^2\log R+O(1).
\end{align*}
Therefore the quotient tends to \(\beta^{-2}=4/(\alpha-1)^2\), proving sharpness.
\end{proof}

\section{The power-weight spectrum for the swirl rate}

For \(p>1\) define
\[
 Z_{\alpha,p}(t)=\int |F|^p\,d\mu_\alpha,
 \qquad H=|F|^{p/2}.
\]

\begin{theorem}\label{thm:Fidentity}
Let \(p>1\) and \(\alpha>1\).  Then every smooth decaying axisymmetric solution satisfies
\begin{align}
 \frac1p\frac{d}{dt}Z_{\alpha,p}
 &+\nu\frac{4(p-1)}{p^2}\int|\nabla H|^2\,d\mu_\alpha\nonumber\\
 &+\nu\frac{(3-\alpha)(\alpha-1)}p
 \int H^2r^{\alpha-2}\,dr\,dz
 =\frac{\alpha-(2p+1)}p\int U|F|^p\,d\mu_\alpha.
 \label{eq:Fweighted}
\end{align}
In particular, \(\alpha=2p+1\) is the unique power weight for which the radial-strain term vanishes.
\end{theorem}

\begin{proof}
For \(p\ge2\), multiply \eqref{eq:F} by \(|F|^{p-2}F r^\alpha\) and integrate.  For \(1<p<2\), use instead the multiplier \(\Phi_\varepsilon'(F)r^\alpha/p\) from \Cref{lem:diffusion} and pass to the limit.  In either case, by \Cref{lem:weighteddiv}, the transport contributes
\[
 -\frac{\alpha-1}{p}\int U|F|^p\,d\mu_\alpha.
\]
The explicit stretching term contributes \(-2\int U|F|^p\,d\mu_\alpha\).  Moving the transport contribution to the right gives the coefficient
\[
 \frac{\alpha-1}{p}-2=\frac{\alpha-(2p+1)}p.
\]
The diffusion terms are exactly \Cref{lem:diffusion}.
\end{proof}

\begin{remark}\label{rem:parameter-dictionary}
The parameter \(\alpha\) translates directly into two standard weighted-swirl parametrizations.  Since the three-dimensional volume element is \(2\pi r\,dr\,dz\),
\[
 \|r^d u^\theta\|_{L^p(\mathbb R^3)}^p
 \simeq \int |F|^p r^{p(d+1)+1}\,dr\,dz.
\]
Thus
\begin{equation}\label{eq:d-dictionary}
 \alpha=p(d+1)+1,
 \qquad
 d=\frac{\alpha-p-1}{p}.
\end{equation}
The critical line \(\alpha=2p+1\) is therefore exactly \(d=1\), the endpoint at which the Hardy-based method in \cite{ChenFangZhang2017} is reported to fail.  Likewise, for the weighted circulation variable of \cite{RencZaj2019},
\[
 u_a=\frac{\Gamma}{r^a}=r^{1-a}u^\theta,
 \qquad
 \alpha=2p+1-ap.
\]
Their scale-balanced condition \(ap=3\) is exactly our scale-invariant line \(\alpha=2p-2\).  Hence the two distinguished lines in the phase diagram already have familiar interpretations in the earlier weighted-swirl notation; what is added below is their sharp coercive separation and endpoint classification.
\end{remark}

Define the full diffusion form
\[
 \mathcal D_{\alpha,p}[H]
 =\frac{4(p-1)}{p^2}\int |\nabla H|^2r^\alpha\,dr\,dz
 +\frac{(3-\alpha)(\alpha-1)}p
 \int H^2r^{\alpha-2}\,dr\,dz.
\]

\begin{theorem}\label{thm:phase}
Fix \(p>1\).
\begin{enumerate}[label=\textnormal{(\roman*)}]
\item If \(1<\alpha\le3\), then
\[
 \mathcal D_{\alpha,p}[H]
 \ge \frac{4(p-1)}{p^2}\int|\nabla H|^2r^\alpha\,dr\,dz.
\]
The coefficient is sharp.
\item If \(3<\alpha<2p+1\), then
\begin{equation}\label{eq:radgap}
 \mathcal D_{\alpha,p}[H]
 \ge c_{\alpha,p}\int|\nabla H|^2r^\alpha\,dr\,dz,
 \qquad
 c_{\alpha,p}=\frac{4(2p+1-\alpha)}{p^2(\alpha-1)}.
\end{equation}
The constant \(c_{\alpha,p}\) is sharp.
\item If \(\alpha=2p+1\), then \(\mathcal D_{\alpha,p}\ge0\), but there is no constant \(c>0\) such that
\[
 \mathcal D_{2p+1,p}[H]\ge c\int|\nabla H|^2r^{2p+1}\,dr\,dz
\]
for all smooth compactly supported \(H\).
\item If \(\alpha>2p+1\), the full form \(\mathcal D_{\alpha,p}\) is indefinite: there exists smooth compactly supported \(H\) with \(\mathcal D_{\alpha,p}[H]<0\).
\end{enumerate}
Therefore no power weight can both cancel the radial-strain term in \eqref{eq:Fweighted} and retain a positive uniform coercive gap in the full weighted gradient of \(|F|^{p/2}\).
\end{theorem}

\begin{proof}
For \(1<\alpha\le3\), the zero-order coefficient is nonnegative, giving (i).  Sharpness follows by taking a fixed nonzero compact radial cutoff and increasingly oscillatory compactly supported functions in the \(z\)-variable; the axial derivative then dominates all lower-order terms, so the quotient tends to \(4(p-1)/p^2\).

Assume now \(\alpha>3\).  Hardy's inequality gives
\[
 \int H^2r^{\alpha-2}\le\frac4{(\alpha-1)^2}
 \int|H_r|^2r^\alpha.
\]
Since the zero-order coefficient is negative,
\begin{align*}
 \mathcal D_{\alpha,p}[H]
 &\ge
 \frac{4(p-1)}{p^2}\int|H_z|^2r^\alpha
 +\frac{4(2p+1-\alpha)}{p^2(\alpha-1)}
  \int|H_r|^2r^\alpha.
\end{align*}
For \(\alpha\ge3\), the second coefficient is no larger than the first because
\[
 \frac{4(p-1)}{p^2}-c_{\alpha,p}
 =\frac{4(\alpha-3)}{p(\alpha-1)}\ge0.
\]
This proves \eqref{eq:radgap} for \(3<\alpha<2p+1\) and nonnegativity at the endpoint.

To prove sharpness and the endpoint statements, let \(h_R\) be the radial Hardy near-extremizers from \Cref{lem:hardy}, and set
\[
 H_{R,L}(r,z)=h_R(r)\eta(z/L),
 \qquad 0\ne\eta\in C_c^\infty(\R).
\]
The radial derivative and zero-order pieces are multiplied by \(L\int\eta^2\), while the axial derivative is multiplied by \(L^{-1}\int|\eta'|^2\).  On the Hardy plateau one has
\[
 \int h_R^2r^\alpha\,dr=O(R^2),
 \qquad
 \int |h_R'|^2r^\alpha\,dr\asymp\log R.
\]
Choosing, for example, \(L=R\), makes the axial derivative negligible relative to the radial derivative as \(R\to\infty\).  The quotient of the full form by the full gradient therefore tends to
\[
 \frac{4(2p+1-\alpha)}{p^2(\alpha-1)}.
\]
This proves sharpness in (ii), loss of every positive full-gradient gap in (iii), and negativity of the full diffusion form for all sufficiently large \(R\) when \(\alpha>2p+1\).
\end{proof}

\begin{corollary}\label{cor:ratio}
For \(3\le\alpha<2p+1\), the magnitude of the inward-compression coefficient in \eqref{eq:Fweighted} divided by the sharp radial coercive coefficient \eqref{eq:radgap} is
\begin{equation}\label{eq:ratio}
 \frac{(2p+1-\alpha)/p}{c_{\alpha,p}}
 =\frac{p(\alpha-1)}4.
\end{equation}
Hence, among powers \(\alpha\ge3\) (the range in which the coupled \(G\)-source pairing has no negative radial factor), moving the weight toward the stretching-cancellation endpoint never improves the compression-to-coercivity ratio; the ratio is minimized at \(\alpha=3\).
\end{corollary}

\begin{proof}
This is direct algebra using \eqref{eq:radgap}.  The right-hand side of \eqref{eq:ratio} is strictly increasing in \(\alpha\).
\end{proof}

\begin{corollary}\label{cor:F-formbound}
Let \(3\le\alpha<2p+1\) and suppose the inward radial strain obeys the weighted form bound
\begin{equation}\label{eq:F-formbound}
 \int U_-H^2r^\alpha\,dr\,dz
 \le \Lambda\int|H_r|^2r^\alpha\,dr\,dz,
 \qquad U_-=(-U)_+.
\end{equation}
Then the compressive term in \eqref{eq:Fweighted} is strictly absorbed by radial diffusion whenever
\begin{equation}\label{eq:F-formbound-threshold}
 \Lambda<\frac{4\nu}{p(\alpha-1)}.
\end{equation}
Among Hardy-coercive powers \(3\le\alpha<2p+1\), the admissible constant in \eqref{eq:F-formbound-threshold} is maximal at \(\alpha=3\), where it equals \(2\nu/p\).
\end{corollary}

\begin{proof}
For \(\alpha<2p+1\), the right-hand side of \eqref{eq:Fweighted} has dangerous part at most
\[
 \frac{2p+1-\alpha}{p}\Lambda
 \int |H_r|^2r^\alpha\,dr\,dz.
\]
The sharp radial diffusion gap is \(\nu c_{\alpha,p}\) from \eqref{eq:radgap}.  The inequality
\((2p+1-\alpha)\Lambda/p<\nu c_{\alpha,p}\) is exactly \eqref{eq:F-formbound-threshold}.  Its right-hand side decreases strictly with \(\alpha\).
\end{proof}

\begin{corollary}\label{cor:scaling-gap}
Under the Navier--Stokes scaling,
\(F_\lambda(r,z,t)=\lambda^2F(\lambda r,\lambda z,\lambda^2t)\), the weighted quantity \(Z_{\alpha,p}\) scales as
\[
 Z_{\alpha,p}[F_\lambda](t)
 =\lambda^{2p-\alpha-2}Z_{\alpha,p}[F](\lambda^2t).
\]
Thus the scale-invariant power is \(\alpha_{\rm sc}=2p-2\), whereas the stretching-free power is \(\alpha_{\rm can}=2p+1\).  They are separated by exactly three radial powers:
\[
 \alpha_{\rm can}-\alpha_{\rm sc}=3.
\]
In particular, no scale-invariant power-weighted \(F^p\) norm can also cancel the radial strain.
\end{corollary}

\begin{proof}
The change of variables \(r'=\lambda r\), \(z'=\lambda z\) gives the scaling factor \(\lambda^{2p}\lambda^{-\alpha}\lambda^{-2}\).  The remaining assertions are algebraic.
\end{proof}

\section{Conjugation to weighted circulation energies}

The complete power-weight family has a simple circulation interpretation, not only at the cancellation endpoint.  This change of variables also explains why the endpoint itself is classical rather than a new a priori estimate.  Put
\[
 \beta=\alpha-2p,
 \qquad Q=|\Gamma|^{p/2}.
\]
Since \(\Gamma=r^2F\),
\begin{equation}\label{eq:general-energy-conjugate}
 \int |F|^pr^\alpha\,dr\,dz
 =\int |\Gamma|^pr^\beta\,dr\,dz.
\end{equation}

\begin{theorem}\label{thm:Gammaweighted}
Let \(p>1\), \(\alpha>1\), and \(\beta=\alpha-2p\).  Under the same finiteness assumptions as in \Cref{thm:Fidentity},
\begin{align}
 \frac1p\frac d{dt}\int|\Gamma|^pr^\beta
 &+\nu\frac{4(p-1)}{p^2}\int|\nabla Q|^2r^\beta\nonumber\\
 &-\nu\frac{(\beta-1)(\beta+1)}p
   \int Q^2r^{\beta-2}
 =\frac{\beta-1}{p}\int U|\Gamma|^pr^\beta.
 \label{eq:Gammaweighted-general}
\end{align}
The identity is exactly equivalent to \eqref{eq:Fweighted} under \(\Gamma=r^2F\).  In particular, cancellation of radial strain occurs precisely when \(\beta=1\), i.e. when \(\alpha=2p+1\).
\end{theorem}

\begin{proof}
For \(p\ge2\), multiply the circulation equation \eqref{eq:Gamma} by \(|\Gamma|^{p-2}\Gamma r^\beta\); for \(1<p<2\), use the same convex renormalization as in \Cref{lem:diffusion}.  The weighted divergence of the meridional velocity is \((\beta-1)U\), so transport contributes \(-(\beta-1)/p\) times the weighted \(|\Gamma|^p\)-moment.  For diffusion,
\[
 \int |\Gamma|^{p-2}\Gamma
 \left(\Gamma_{rr}-\frac1r\Gamma_r+\Gamma_{zz}\right)r^\beta
 =-\frac{4(p-1)}{p^2}\int|\nabla Q|^2r^\beta
 +\frac{(\beta-1)(\beta+1)}p\int Q^2r^{\beta-2}.
\]
The boundary term at the axis vanishes even when \(\beta\le1\), because Cartesian regularity gives \(\Gamma=O(r^2)\) and hence
\(r^{\beta-1}|\Gamma|^p=O(r^{\alpha-1})\to0\).  Rearrangement gives \eqref{eq:Gammaweighted-general}.  Substituting \(\beta=\alpha-2p\) and \(\Gamma=r^2F\) recovers \eqref{eq:Fweighted}.
\end{proof}

\section{The cancellation endpoint is the physical-space circulation norm}

The previous endpoint has a concrete interpretation.  For \(p\ge2\), the resulting circulation estimate belongs to the classical a priori theory; see, for example, Chae--Lee \cite{ChaeLee2002}.  The same identity for \(1<p<2\) follows from the convex renormalization above.  We include the calculation because the sharp power-weight phase diagram lands on this circulation energy exactly.

\begin{theorem}\label{thm:circulationendpoint}
Let \(p>1\) and set \(\alpha=2p+1\).  Then
\begin{equation}\label{eq:energy-conjugate}
 \int |F|^p r^{2p+1}\,dr\,dz
 =\int |\Gamma|^p r\,dr\,dz,
\end{equation}
and
\begin{align}
 &\frac{4(p-1)}{p^2}\int|\nabla |F|^{p/2}|^2r^{2p+1}\,dr\,dz
 -4(p-1)\int |F|^p r^{2p-1}\,dr\,dz\nonumber\\
 &\hspace{35mm}=\frac{4(p-1)}{p^2}
 \int\left|\nabla |\Gamma|^{p/2}\right|^2r\,dr\,dz.
 \label{eq:diss-conjugate}
\end{align}
Consequently \eqref{eq:Fweighted} at \(\alpha=2p+1\) is exactly
\begin{equation}\label{eq:GammaLp}
 \frac1p\frac d{dt}\int|\Gamma|^p r\,dr\,dz
 +\nu\frac{4(p-1)}{p^2}
 \int\left|\nabla|\Gamma|^{p/2}\right|^2r\,dr\,dz=0,
\end{equation}
the classical \(L^p\) circulation identity.
\end{theorem}

\begin{proof}
Since \(\Gamma=r^2F\), \eqref{eq:energy-conjugate} is immediate.  Let \(H=|F|^{p/2}\).  Then \(|\Gamma|^{p/2}=r^pH\).  Expanding,
\begin{align*}
 \int \left|\partial_r(r^pH)\right|^2r\,dr\,dz
 &=\int |H_r|^2r^{2p+1}\,dr\,dz
 -p^2\int H^2r^{2p-1}\,dr\,dz,
\end{align*}
where the cross term is integrated by parts.  Also
\[
 \int \left|\partial_z(r^pH)\right|^2r\,dr\,dz
 =\int |H_z|^2r^{2p+1}\,dr\,dz.
\]
Multiplying by \(4(p-1)/p^2\) gives \eqref{eq:diss-conjugate}.  At \(\alpha=2p+1\) the compression coefficient in \eqref{eq:Fweighted} vanishes, so \eqref{eq:GammaLp} follows.  One obtains the same identity directly by testing \eqref{eq:Gamma} with \(|\Gamma|^{p-2}\Gamma r\); the axis contribution vanishes because \(\Gamma(0,z,t)=0\).
\end{proof}

\begin{remark}
For \(p=2\), \(\alpha=5\) is therefore not a new stretching-free \(F\)-energy: it is exactly the three-dimensional \(L^2\) circulation norm \(\int\Gamma^2r\).  At the other end, the formal limit \(p\downarrow1\) sends the cancellation weight to \(\alpha=3\), the five-dimensional conservative weight.
\end{remark}

\section{The weighted spectrum of the vorticity ratio}

Set \(W=\Gamma_z/r\).  Since \(F=\Gamma/r^2\),
\begin{equation}\label{eq:sourcefactor}
 \partial_z(F^2)=\frac{2\Gamma W}{r^3}.
\end{equation}

\begin{theorem}\label{thm:Gidentity}
For every \(\alpha>1\),
\begin{align}
 \frac12\frac d{dt}\int G^2\,d\mu_\alpha
 &+\nu\int|\nabla G|^2\,d\mu_\alpha
 +\nu\frac{(3-\alpha)(\alpha-1)}2
 \int G^2r^{\alpha-2}\,dr\,dz\nonumber\\
 &=\frac{\alpha-1}{2}\int UG^2\,d\mu_\alpha
 +2\int \Gamma WG\,r^{\alpha-3}\,dr\,dz.
 \label{eq:Gweighted}
\end{align}
The full \(G\)-diffusion form is coercive for \(1<\alpha<5\).  More precisely, for \(1<\alpha\le3\),
\[
 \int|\nabla G|^2r^\alpha
 +\frac{(3-\alpha)(\alpha-1)}2\int G^2r^{\alpha-2}
 \ge \int|\nabla G|^2r^\alpha,
\]
while for \(3<\alpha<5\),
\begin{equation}\label{eq:Ggap}
 \int|\nabla G|^2r^\alpha
 +\frac{(3-\alpha)(\alpha-1)}2\int G^2r^{\alpha-2}
 \ge\frac{5-\alpha}{\alpha-1}\int|\nabla G|^2r^\alpha.
\end{equation}
Both coefficients are sharp.  At \(\alpha=5\) the full form is nonnegative but has no positive full-gradient gap, and for \(\alpha>5\) the full diffusion form is indefinite.
\end{theorem}

\begin{proof}
Multiply \eqref{eq:G} by \(G r^\alpha\).  The transport contribution is given by \Cref{lem:weighteddiv}; \Cref{lem:diffusion} with \(p=2\) gives the diffusion terms; and \eqref{eq:sourcefactor} gives the final source pairing.  For \(1<\alpha\le3\), the zero-order term is nonnegative.  The coefficient one in front of the full gradient is sharp: with a fixed compact radial factor, let the compact axial factor oscillate at frequency \(N\to\infty\); then the axial derivative dominates the zero-order and radial pieces.  For \(3<\alpha<5\), Hardy's inequality gives
\[
 1+\frac{2(3-\alpha)}{\alpha-1}=\frac{5-\alpha}{\alpha-1},
\]
and this coefficient is at most the untouched axial coefficient one.  Sharpness, endpoint loss of the full-gradient gap, and indefiniteness above \(5\) follow from the same tensor-product Hardy near-extremizers used in \Cref{thm:phase}.
\end{proof}

\begin{corollary}\label{cor:Gopt}
Among powers \(3\le\alpha<5\), the ratio of the explicit compression coefficient \((\alpha-1)/2\) in \eqref{eq:Gweighted} to the sharp radial gap in \eqref{eq:Ggap} is
\[
 \frac{(\alpha-1)^2}{2(5-\alpha)}.
\]
It is strictly increasing on \([3,5)\) and equals \(1\) at \(\alpha=3\).
\end{corollary}

\begin{corollary}\label{cor:G-formbound}
Let \(3\le\alpha<5\).  If the expansive strain satisfies
\[
 \int U_+G^2r^\alpha\,dr\,dz
 \le \Lambda\int|G_r|^2r^\alpha\,dr\,dz,
 \qquad U_+=\max(U,0),
\]
then the explicit compression/transport contribution in \eqref{eq:Gweighted} is strictly absorbed by the radial diffusion gap whenever
\[
 \Lambda<\frac{2\nu(5-\alpha)}{(\alpha-1)^2}.
\]
The admissible constant is maximal at \(\alpha=3\), where it equals \(\nu\).  This assertion concerns only the transport/compression term; the source \(2\int\Gamma WG r^{\alpha-3}\) remains to be estimated separately.
\end{corollary}

\begin{proof}
The positive contribution of \(U\) in \eqref{eq:Gweighted} is bounded by \((\alpha-1)\Lambda/2\) times the radial derivative energy, while \eqref{eq:Ggap} supplies \(\nu(5-\alpha)/(\alpha-1)\).  Comparing the two coefficients gives the stated threshold.  The function \(2(5-\alpha)/(\alpha-1)^2\) decreases on \([3,5)\).
\end{proof}

\begin{remark}\label{rem:threeweights}
For the quadratic swirl-rate energy \(p=2\), three powers are especially transparent:
\begin{center}
\begin{tabular}{ccl}
\toprule
weight & role & exact interpretation \\
\midrule
\(\alpha=2\) & scale invariant & critical weighted \(F^2\) size \\
\(\alpha=3\) & five-dimensional & physical swirl kinetic energy; pure \(\Delta_5\) Dirichlet form \\
\(\alpha=5\) & stretching free & three-dimensional \(L^2\) circulation norm; Hardy-critical in \(F\) \\
\bottomrule
\end{tabular}
\end{center}
The separation of these three roles is one reason a single power multiplier cannot close the large-swirl problem by itself.
\end{remark}

\section{Why the five-dimensional weight is distinguished}

The preceding identities give an intrinsic characterization of \(r^3drdz\).

\begin{theorem}\label{thm:triple}
Among power weights \(d\mu_\alpha=r^\alpha drdz\) with \(\alpha>1\), the exponent \(\alpha=3\) is the unique value for which the following three properties hold simultaneously:
\begin{enumerate}[label=\textnormal{(\roman*)}]
\item the \(F\)-equation is conservative with respect to \(\mu_\alpha\), namely
\[
 F_t+\operatorname{div}_{\mu_\alpha}(bF)=\nu\Da F;
\]
\item \(\Da\) is formally self-adjoint in \(L^2(d\mu_\alpha)\);
\item the \(G\)-source pairing has no radial weight,
\[
 \int G\,\partial_z(F^2)\,d\mu_\alpha
 =2\int\Gamma WG\,dr\,dz.
\]
\end{enumerate}
Moreover, among the powers \(3\le\alpha<5\) for which the \(G\)-source pairing has no negative radial factor and the full diffusion retains a positive Hardy gap, \(\alpha=3\) uniquely minimizes the coefficient ratio of \Cref{cor:Gopt}, equivalently maximizing the absorbable form-bound constant in \Cref{cor:G-formbound}.  For the \(F^p\) family it likewise minimizes \Cref{cor:ratio} and maximizes the threshold in \Cref{cor:F-formbound} over the Hardy-coercive range \(3\le\alpha<2p+1\).
\end{theorem}

\begin{proof}
By \Cref{lem:weighteddiv},
\[
 F_t+\operatorname{div}_{\mu_\alpha}(bF)
 =\nu\Da F+(\alpha-3)UF,
\]
so (i) holds if and only if \(\alpha=3\).  The radial part of \(\Da\) is
\[
 \partial_{rr}+\frac3r\partial_r
 =r^{-3}\partial_r(r^3\partial_r),
\]
which is formally self-adjoint exactly in the measure \(r^3dr\); thus (ii) also selects \(\alpha=3\).  Finally, \eqref{eq:sourcefactor} shows that the source pairing carries the factor \(r^{\alpha-3}\), so it is flat precisely at \(\alpha=3\).  The optimization statements are \Cref{cor:ratio,cor:Gopt}.
\end{proof}

\section{Consequences and limitations of power-weight methods}

The preceding phase diagram can be summarized as follows.  For each \(p>1\), the range \(\alpha\ge3\), in which the coupled source pairing carries no negative radial factor, contains three regimes for the \(F^p\) energy:
\[
 \boxed{
 3\le\alpha<2p+1:
 \ \text{Hardy-coercive diffusion + uncancelled compression};}
\]
\[
 \boxed{
 \alpha=2p+1:
 \ \text{compression cancellation + Hardy-critical diffusion};}
\]
\[
 \boxed{
 \alpha>2p+1:
 \ \text{indefinite full diffusion form}.}
\]
The endpoint is not an unknown new energy: it is exactly the \(L^p\) energy of \(\Gamma\).  Thus, within the class of pure power multipliers, there is no intermediate weight that both eliminates the radial strain from the swirl-rate equation and retains a strict derivative gap in \(F\).

For \(p=2\), the two most natural weights are particularly transparent.  The five-dimensional weight \(\alpha=3\) gives
\[
 \int F^2r^3=\int (u^\theta)^2r,
\]
the physical swirl kinetic energy (up to \(2\pi\)), and \(\Da\) contributes the pure five-dimensional Dirichlet form.  The cancellation weight \(\alpha=5\) gives
\[
 \int F^2r^5=\int\Gamma^2r,
\]
the three-dimensional \(L^2\) circulation norm; the \(F\)-Hardy gap is then exactly critical, while the conjugated \(\Gamma\)-gradient remains dissipative.  This identifies the endpoint degeneration as a change of natural variable rather than a loss of viscosity in the underlying Navier--Stokes equation.

The result is a limitation of one specific strategy, not a regularity obstruction for the PDE itself.  Non-power weights, coupled energies, local geometry, sign information, and nonlocal structure can contain information absent from the present family.  Earlier work already contains the individual ingredients of weighted swirl energies, circulation estimates, and endpoint Hardy restrictions \cite{ChaeLee2002,ChenFangZhang2017,RencZaj2019}; other regularity arguments use form boundedness, Hardy--Sobolev estimates, Lorentz control, and geometric information in ways not reducible to a single power-weight energy \cite{LeiZhang2017,Zhang2017,WangHuangWeiYu2024}.  The phase diagram organizes these familiar ingredients into a sharp two-parameter statement and identifies exactly what the pure power-multiplier mechanism can and cannot buy.

\end{document}